\documentclass[article,12pt]{article}
\usepackage[a4paper,textwidth=18cm,textheight=22cm,includehead]{geometry}

\usepackage[latin1]{inputenc}

\usepackage{fancyhdr}

\usepackage{hyperref}
\usepackage{latexsym}
\usepackage{amsthm}
\usepackage{epic}
\usepackage{amsfonts}
\usepackage{ytableau}

\usepackage{rotating}

\usepackage[T1]{fontenc}

\usepackage{amsmath} 
\usepackage{amssymb} 
\usepackage[all]{xy}  
\usepackage{graphicx}  
\usepackage{xcolor}
\usepackage{mathrsfs}
\usepackage{euscript}
\usepackage{listings}

\newtheorem{thm}{Theorem}[section]
\newtheorem{prop}[thm]{Proposition}

\newtheorem{lem}[thm]{Lemma}
\newtheorem{defn}[thm]{Definition}
\newtheorem{example}[thm]{Example}
\newtheorem{notation}[thm]{Notation}
\newtheorem{remark}[thm]{Remark}

\newcommand{\G}{\mathbb{G}}
\newcommand{\Pe}{\mathbb{P}}
\newcommand{\K}{\mathbb{K}}
\newcommand{\C}{\mathbb{C}}
\newcommand{\Z}{\mathbb{Z}}
\newcommand{\Sc}{\mathbb{S}}
\newcommand{\QQ}{\mathcal{Q}}
\newcommand{\OO}{\mathcal{O}}
\newcommand{\s}{\mathcal{S}}
\newcommand{\dual}{^\vee}
\newcommand{\F}{\mathcal{F}}

\newcommand{\sym}{\mathfrak{S}}

\makeatletter
      \def\@setcopyright{}
      \def\serieslogo@{}
      \makeatother

\begin{document}

\author{
Enrique Arrondo$^{1}$, Alicia Tocino$^{2}$\\ \\
\small{$^{1}$arrondo@mat.ucm.es}\\
\small{ Universidad Complutense de Madrid, Departamento de \'Algebra,}\\ 
\small{Ciudad Universitaria, Plaza Ciencias, 3, 28040 Madrid, Spain}\\
\small{$^{2}$aliciatocinosanchez@ucm.es}\\ 
\small{ Universidad de M\'alaga, Departamento de \'Algebra, Geometr\'a y Topolog\'ia,}\\ 
\small{Bulevar Louis Pasteur, 31, 29010 M\'alaga, Spain}}
\date{July 2019}

\title{Cohomological characterization of Universal Bundles of $\G(1,n)$}

\maketitle

\begin{abstract}
We characterize directs sums of twists of symmetric powers of the universal quotient  bundle over the Grassmannian of lines. We use a method that could be used for analogue results on any arbitrary variety, and that should give stronger results than using the standard technique of Beilinson's spectral sequence.

\end{abstract}
\fancypagestyle{plain}{
\fancyhead[L]{}
\fancyhead[C]{}
\fancyhead[R]{}
\fancyfoot[L]{ This research was partly supported by an FPU grant of the University Complutense of Madrid (2011-2015) and a Spanish government research project, number MTM2012-32670}
\fancyfoot[C]{}
\fancyfoot[R]{}
\renewcommand{\headrulewidth}{0pt}
\renewcommand{\footrulewidth}{0.5pt}
}
\pagestyle{fancy}

\section{Introduction}
One of the key results to understand the structure of vector bundles over a projective space is  Horrocks' criterion (see \cite{HOR}): \textit{A vector bundle $F$ over $\Pe^n$ splits as a direct sum of line bundles if and only if it does not have intermediate cohomology (i.e. $H^j_*(F)=0$ for $j=1,2,\ldots,n-1$).} There has been many proofs and generalizations of this theorem, using many different techniques, like restriction to hyperplanes, regularity theory or Beilinson theorem (in the general framework of derived categories). One interpretation of Horrocks' theorem  is to regard it as a characterization of arithmetically Cohen-Macaulay (aCM for short) vector bundles (i.e. without intermediate cohomology) over the projective space. 

In this paper we will deal with the second point of view of Horrocks' theorem, namely as a criterion to characterize vector bundles that split as a direct sum of line bundles. Such a characterization has been extended to other projective varieties. For example, G. Ottaviani characterized such vector bundles over Grassmannians and smooth quadrics (see \cite{OTT2} and \cite{OTT} and Theorem \ref{teorema ottaviani split} for the particular case of Grassmannian of lines). 

An improvement of Ottaviani's criterion, which will be the starting point of this paper was given by E. Arrondo and F. Malaspina in \cite{AM} (Theorem \ref{teorema de AM}). The idea behind their proof (and similar ones in that paper) is quite simple: if a vector bundle $F$ has a direct summand $\OO(l)$, this is equivalent to having maps $\OO(l)\to F$ and $F\to\OO(l)$ whose composition is not zero. The tricky part is to relate the composition pairing $Hom(\OO(l),F)\times Hom(F,\OO(l))\to Hom(\OO(l),\OO(l))$ with the perfect pairing given by Serre duality.

As it is also done in \cite{AM}, one could be interested in characterizing cohomologically other vector bundles, and not only line bundles. For instance, in $2005$, L. Costa and R. M. Mir\'o-Roig gave a characterization, for a given Schur functor, of the vector bundle $\Sc_\lambda \QQ$, where $\QQ$ is the universal quotient bundle of rank $k+1$ over the  Grassmannian $\G(k,n)$ (see \cite{MiroRoig2} and \cite{MiroRoig}). This characterization is done in terms of the other universal vector bundle (of rank $n-k$) and characterizes precisely $\oplus \Sc_\lambda\QQ$ but not its twists. However we will be interested in using the same universal bundle $\QQ$ for the characterization, and we also want the characterization to be up to a twist, as it happens for Horrocks' theorem.

A result of the type we are looking for is the following by E. Arrondo and B. Gra\~na (see \cite{AG}):
\textit{Let $F$ be a vector bundle over $\G(1,4)$. Then $F$ can be expressed as 
$\oplus\OO(l_{i_0})\bigoplus\oplus\QQ(l_{i_1})$ if and only if
neither $F$ nor $F\otimes \QQ$ do not have intermediate cohomology.} The idea for obtaining such result was as follows. First, one looks at Theorem \ref{teorema ottaviani split}, which provides a characterization of the direct sums of line bundles. Since $\QQ$ does not split, there must be some hypothesis of that theorem that it does not satisfy. There is precisely one of them, and one manage to prove (by techniques that are not relevant here) that the rest of the hypotheses work to characterize directs sums of twists of $\OO$ and $\QQ$. In fact, the authors continue to remove conditions until getting some description (very far from an actual classification) of aCM vector bundles over $\G(1,4)$.

In this paper, we want to give a cohomological characterization for direct sums of twists of symmetric product of the quotient bundle $\QQ$ in the Grassmannian of lines $\G(1,n)$. This is the main result of the second author's PhD thesis \cite{Toc}. We will give such a classification for symmetric powers of order not bigger than $n-2$. The original reason was that a symmetric power $S^k\QQ$ is aCM if and only if $k\le n-2$, so that our result could be a help to understand aCM vector bundles over $\G(1,n)$. However such restriction is needed anyway, since our proof has some obstruction when $k>n-2$. 

As main ideas, we will use those from \cite{AG} and \cite{AM} that we already pointed out. Specifically, we will start from the splitting criterion of Theorem \ref{teorema de AM} (which is stronger than the one of  Theorem \ref{teorema ottaviani split}). Contrary to what happened in the proof of the characterization given by E. Arrondo and B. Gra\~na, removing the only hypothesis of Theorem \ref{teorema de AM} will not be now sufficient to characterize direct sums of twists of $\OO$ and $\QQ$, so we will need to add more hypotheses for such characterization. As in  \cite{AM}, the idea to find direct summand of $\QQ(l)$ will be to relate it with Serre duality.
With this ideas in mind, one can go on, removing in each step one condition and adding few ones, until arriving to the wanted classification.

We start the paper with a section of preliminaries, in which we recall the main results about Grassmannians, their universal bundles, the notation of Schur functors and Bott's algorithm. In section 3, we explain the general method by giving the splitting criterion for Grassmannians of lines. In section 4 we prove our main result. We finish with a section of remarks, in which we show how our method gives much stronger results than derived categories, explain how our method will work in general, and propose the way to extend it to an arbitrary Grassmannian.

\textbf{Acknowledgments.} We would like to thank G. Ottaviani for many useful conversations and hints. We also thank A. Kuznetsov, who suggested how to use derived categories to get a similar result.

\section{Preliminaries}
\subsection{Notation and definitions}

Let $V$ be a $(n+1)-$dimensional vector space over an algebraically closed field $\K$ of characteristic zero and let $\Pe^n=\Pe(V)$ be projective space  of all the hyperplanes in $V$, or equivalently we can consider $\Pe^n=\Pe(V)$ as the set of the lines in $V^*$.

We define the \textit{Grassmann variety} or \textit{Grassmannian}, $\G(k,n)=\G(k,\Pe^n)$, as the set consisting of all $k-$dimensional linear subspaces of $\Pe^n$. This variety is naturally identified with the set of $(k+1)-$dimensional linear subspaces of the dual $V^*$ or even with the set of $(k+1)-$dimensional quotients of $V$.

\begin{notation}{\rm
We will use $^*$ for denoting the dual of a vector space and $\dual$ for denoting the dual of a vector bundle. As usual, we will not distinguish between a vector bundle and its locally free sheaf of sections.}

\end{notation}

We consider $\QQ\dual=\{(v,\Lambda)\in V^*\times \G\, |\, v\in \Lambda\}$ and $V^*\times \G/\QQ\dual\simeq \s$ the \textit{universal vector bundles} of $\G(k,n)$ of rank $k+1$ and $n-k$ respectively. We also consider the \textit{universal exact sequence}:
\begin{equation*}\label{sucesionexactauniversal}
\xymatrix{ 
	0 \ar[r] & \s\dual \ar[r]^\varphi & V\otimes \OO \ar[r]^\phi & \QQ \ar[r] & 0}
\end{equation*}

\begin{remark}\label{isomorfismosentrefibradosuniversales}{\rm 
We recall some natural isomorphisms between the universal bundles over the Grassmannians $\G(k,n)$:
\begin{itemize}
	\item $\QQ\dual\simeq \bigwedge^{k}\QQ(-1)$. For the particular case $k=1$ we get $\QQ\dual\simeq \QQ(-1)$ and if we apply the $j-$th symmetric power results $S^j\QQ\dual\simeq(S^j\QQ)(-j)$.
	\item $\bigwedge^{j}\s\dual\simeq \bigwedge^{n-k-j}\s(-1)$ (where $\bigwedge^j$ denotes the $j-$th wedge power).
\end{itemize}  }
\end{remark}

\subsection{Eagon-Northcott complexes}
The following exact sequence is the \textit{Eagon-Northcott complexes} associated to the universal exact sequence $\G(k,n)$:\begin{equation}
\resizebox{\textwidth}{!}{
\xymatrix{ 
0\longrightarrow S^j \QQ\dual\longrightarrow  V^* \otimes S^{j-1}\QQ\dual \longrightarrow \bigwedge^2 V^*\otimes S^{j-2}\QQ\dual\longrightarrow\ldots \longrightarrow \bigwedge^{j-1} V^*\otimes\QQ\dual \longrightarrow  \bigwedge^{j}V^*\otimes \OO \longrightarrow \bigwedge^j\s \longrightarrow 0
}} \tag{$R\dual_j$}
\end{equation}
By dualizing $(R_j\dual)$ we give its dual complex $(R_j)$:
\begin{equation}
\resizebox{\textwidth}{!}{
\xymatrix{ 
0\longrightarrow \bigwedge^j \s\dual\longrightarrow \bigwedge^j V \otimes\OO \longrightarrow \bigwedge^{j-1}V\otimes\QQ \longrightarrow \ldots \longrightarrow \bigwedge^2 V\otimes S^{j-2}\QQ \longrightarrow  V\otimes S^{j-1}\QQ\longrightarrow S^j\QQ \longrightarrow 0
}} \tag{$R_j$}
\end{equation}
We use them often throughout the paper.

\subsection{Schur Functors}
We introduce the notion of Young diagram, its corresponding irreducible representation and also its corresponding Schur functor (in particular we define $\Sc_\lambda\QQ$). See \cite{FH} and \cite{WEY}.

\begin{defn}\label{def young1}
\textnormal{To a partition $\lambda$ is associated a \textit{Young diagram} denoted by
$\lambda=(\lambda_1,\lambda_2,\ldots,\lambda_m)$, where $\lambda_1\geq\lambda_2\geq\ldots\geq\lambda_m\geq 0$. It 
consists of a collection of boxes ordered in 
consecutive rows, where the $i-$th row has exactly $\lambda_i$ boxes (the rows of boxes are aligned top-left). The number of boxes of $\lambda$ is denoted by 
$|\lambda|=\lambda_1+\lambda_2+\ldots+\lambda_m$. }
\end{defn}

\begin{defn}\label{def young3}
\textnormal{For a given Young diagram we can number the boxes. Any filling of $\lambda$ with numbers is called a \textit{Young tableau}.}
\end{defn} 
Young diagrams can be used to describe projection operators for the regular representation, which will then give the irreducible representations of $\sym_d$.
Just to fix convention, for a given Young diagram, number 
the boxes consecutively  (from left to right and top to bottom). Here we use all numbers from $1$ to $d$ in order to fill $d$ boxes. 
More generally, a tableau can allow repetitions
of numbers. Each filling describes a vector in $V^{\otimes d}$.	
\begin{equation*}
\ytableausetup
{mathmode, boxsize=1.25em}
\begin{ytableau}
 {1} & {2} & {3} & {4} & {5} \\
 {6} & {7} & {8}  \\
 {9} & {10} & {11} \\
 {12} & {13}   \\
 {14}
\end{ytableau}
\end{equation*}
Given a tableau, say the canonical one shown, define two subgroups of the symmetric group (due to the filling, we can consider the elements of $\sym_d$ as permuting the boxes and the representations constructed will be isomorphic to the ones made with the canonical tableau):
$$P=P_\lambda=\{g\in\sym_d:\, g \text{ preserves each row}\}$$
and
$$Q=Q_\lambda=\{g\in\sym_d:\,g \text{ preserves each column}\}$$
They depend on $\lambda$ but also on the filling of $\lambda$.
In the group algebra $\C\sym_d$, we introduce two elements corresponding to the subgroups and we set
$$a_\lambda=\sum_{g\in P}e_g\quad\text{ and }\quad b_\lambda=\sum_{g\in Q}sgn(g)\cdot e_g$$
To see what $a_\lambda$ and $b_\lambda$ do, observe that if $V$ is any vector space and $\sym_d$ acts on the $d-$th tensor power $V^{\otimes d}$ by permuting factors, the image of the element $a_\lambda\in\C\sym_d\to End(V^{\otimes d})$ is just the subspace
$$Im(a_\lambda)=S^{\lambda_1}V\otimes S^{\lambda_2}V\otimes\cdots\otimes S^{\lambda_k}V\subset V^{\otimes d}$$
where the inclusion on the right is obtained by grouping the factors of $V^{\otimes d}$ according to the rows of the Young tableau. Similarly, the image of $b_\lambda$ on this tensor power is
$$Im(b_\lambda)=\bigwedge^{\mu_1}V\otimes\bigwedge^{\mu_2}V\otimes\ldots\otimes\bigwedge^{\mu_r}V\subset V^{\otimes d}$$
where $\mu$ is the conjugate partition to $\lambda$.
This is because $\sym_d$ acts on the $d-$th tensor power $V^{\otimes d}$ by permuting factors.

\begin{defn}
\textnormal{Finally, we set $c_\lambda=a_\lambda\cdot b_\lambda\in\C\sym_d$
that is called the \textit{Young symmetrizer} corresponding to $\lambda$.
We denote the image of 
$c_\lambda$ on $V^{\otimes d}$ by $\Sc_\lambda V=Im(c_\lambda|_{V^{\otimes d}})$. We call the functor that associates $V\leadsto \Sc_\lambda V$ the \textit{Schur functor} corresponding to $\lambda$.}
\end{defn}

\begin{example} \label{1} \textnormal{Let us see some examples of Schur functors.
\begin{itemize}
	\item (\textit{Symmetric power}) For the partition $\lambda=(d)$ corresponds to the functor $V\leadsto S^d V$
	\begin{equation*}
	\Sc_{(d)} V=S^d V \Longrightarrow \ytableausetup
	{mathmode, boxsize=1.25em}
	\underbrace{\begin{ytableau}
	{} & {} & {} & \none[\ldots]
	& {} & {}
	\end{ytableau}}_{d \text{ boxes}}
	\end{equation*}
	\item (\textit{Exterior power}) For the partition $\lambda=(\underbrace{1,\ldots,1}_{d})$ corresponds to the functor $V\leadsto \bigwedge^d V$
	\begin{equation*}
	\Sc_{\underbrace{(1,1,\ldots,1)}_{d}} V=\bigwedge^d V \Longrightarrow \ytableausetup
	{mathmode, boxsize=1.25em}
	\left.
	\begin{array}{rrr}
	\begin{ytableau}
		{}\\
		{} \\
		{} \\
		\none[ \vdots]\\
		{}\\
		 {}
		\end{ytableau}
	\end{array}
	\right\}\text{d boxes}
	\end{equation*}
\end{itemize}}
\end{example}

\begin{remark}
\textnormal{Notice that the only Schur functor of a vector space $V$ of dimension $2$ is just a symmetric power with some particular twist. This gives the Schur functor form of the universal bundle $\QQ$ of rank $2$.
\begin{equation*}
S^jV(i)\leadsto
\ytableausetup
{mathmode, boxsize=1.25em}
\underbrace{\begin{ytableau}
{} & {} & {} & {} & {} &\none[\dots]
& {} & {} \\
{} & {} & {} & {} & {} & \none[\dots]
& {} & {}
\end{ytableau}}_{i}
\underbrace{\begin{ytableau}
{} & {} & {} & {} & &\none[\dots]
& {} & {} 
\end{ytableau}}_{j}
\leadsto \Sc_{(i+j,i)}V
\end{equation*}}
\end{remark}

\subsection{Bott's algorithm}
We state the algorithm to compute the cohomology of $\Sc_\lambda\QQ\otimes\Sc_\mu\s\dual$ (or $\Sc_\eta\s\otimes \Sc_\alpha\QQ\dual$) over the Grassmannians $\G(k,n)$. For any partitions $\lambda=(\lambda_0,\ldots,\lambda_k)$ and $\mu=(\mu_{k+1},\ldots,\mu_n)$, we set
$$\nu_{(\lambda,\mu)}:=\Sc_\lambda\QQ\otimes\Sc_\mu\s\dual.$$
By considering $\eta=(\eta_1,\ldots,\eta_{n-k})$ and $\alpha=(\alpha_{n-k+1},\ldots,\alpha_{n+1})$ instead of $\lambda$ and $\mu$ we set
$$\nu_{(\eta,\alpha)}:=\Sc_\eta\s\otimes \Sc_\alpha\QQ\dual.$$
From Remark 4.1.5, Corollary 4.1.7 and Corollary 4.1.9 of \cite{WEY} we get the following algorithm that gives the values of $H^j(\Sc_\lambda\QQ\otimes\Sc_\mu\s\dual)$ (or $H^j(\Sc_\eta\s\otimes \Sc_\alpha\QQ\dual)$) for all $j\geq 0$.

\begin{prop}\label{bott}(Bott's algorithm)
Let $\nu:=(\lambda_0,\ldots,\lambda_k,\mu_{k+1},\ldots,\mu_n)$ and set $j=0$.
We iterate the following two steps:
\begin{enumerate}
	\item If $\nu$ is a partition, then $H^j(\Sc_\lambda\QQ\otimes\Sc_\mu\s\dual)=\Sc_\nu \K^{n+1}$ and $H^i(\Sc_\lambda\QQ\otimes\Sc_\mu\s\dual)=0$
	for all $i\neq j$
	\item If $\nu$ is not a partition, then consider the first $l$ such that $\nu_l<\nu_{l+1}$. Two possibilities can occur:
	\begin{itemize}
		\item If $\nu_{l+1}-\nu_{l}=1$, then $H^i(\Sc_\lambda\QQ\otimes\Sc_\mu\s\dual)=0$ for all $i\geq 0$  
		\item If $\nu_{l+1}-\nu_{l}\neq 1$, then consider $\nu:=(\nu_1,\ldots,\nu_{l-1},\nu_{l+1}-1,\nu_l+1,\nu_{l+2},\ldots,\nu_{n+1})$ and $j=j+1$, and go back to step $(1).$
	\end{itemize}  
\end{enumerate}
\end{prop}


Applying Bott's algorithm we obtain the following statements (we quote \cite{Toc} for details). 
\begin{prop}\label{producto tensorial de productos simetricos}
Consider $\QQ$ the quotient bundle of rank $2$ over the Grassmannian of lines $\G(1,n)$. Then:
\begin{description}
	\item[(i)] $H^j(S^i\QQ(l))=0$ for $l\geq 0$, $i\geq 0$ and $j>0$
	\item[(ii)] $H^j(S^i\QQ(-l))=0$ for $l> 0$, $i\leq n-2$ and $j<2n-2$
\end{description}
\end{prop}
\begin{proof}
See Proposition 2.2.13 of \cite{Toc}.
\end{proof}

\begin{prop}\label{productos simetricos con cohomologia}
Consider $\QQ$ the quotient bundle of rank $2$ over the Grassmannian of lines $\G(1,n)$. Then,
$$H^{n-1}(S^i\QQ(-n-r))\neq 0\quad \text{for}\quad r=0,1,\ldots,i-n+1\quad\text{and}\quad i\geq n-1.$$
\end{prop}
\begin{proof}
See Proposition 2.2.14 of \cite{Toc}.
\end{proof}

These two propositions mean that the only symmetric powers without intermediate cohomology of $\QQ$ over $\G(1,n)$ are the following:
$$\QQ,\,S^2\QQ,\, S^3\QQ,\,\ldots,\, S^{n-3}\QQ,\,S^{n-2}\QQ.$$
Now, one can be interested also in the Schur functors without intermediate cohomology of $\QQ$ in the case of $\G(k,n)$. Let us give as an example the case of $\G(2,5)$.
\begin{example}{\rm
It can be proved using Bott's algorithm that the only Schur functors of $\QQ$ over $\G(2,5)$ without intermediate cohomology are the following:
$$\QQ,\, S^2\QQ,\, \Sc_{(2,1)}\QQ,\,\Sc_{(3,1)}\QQ,\,\Sc_{(4,2)}\QQ$$   }
\end{example}

The following three statements will be used in Proposition \ref{prop}.
\begin{prop}\label{es simple}
Consider $\QQ$ the quotient bundle of rank $2$ over the Grassmannian of lines $\G(1,n)$. Then $S^i\QQ$ is simple (i.e. $Hom(S^i\QQ,S^i\QQ)=\K$).
\end{prop}
\begin{proof}
See Proposition 2.2.16 of \cite{Toc}. It is also a consequence of the stability of $\QQ$.
\end{proof}

\begin{lem}\label{cohom para lema 1}
Consider $\QQ$ the quotient bundle of rank $2$ over the Grassmannian of lines $\G(1,n)$ and $k\in\{1,\ldots,n-2\}$. Then,
\begin{description}
	\item[(i)] $H^{2n-2}(S^{n-k-1}\QQ\dual(-1)\otimes S^{n-k-2}\QQ(-n+1))=0$
	\item[(ii)] $H^{n-1}(S^{n-k-1}\QQ\dual(-1)\otimes S^{k-1}\QQ(-k))=0$
\end{description}
\end{lem}
\begin{proof}
See Lemma 2.2.17 of \cite{Toc}.
\end{proof}

\begin{lem}\label{cohom para lema 3}
Consider $\QQ$ the quotient bundle of rank $2$ over the Grassmannian of lines $\G(1,n)$ and $k\in\{1,\ldots,n-2\}$. Then,
\begin{description}
	\item[(i)] $H^{n-1}(S^k\QQ\dual\otimes S^{n-k-2}\QQ\dual(-1))=0$
	\item[(ii)] $H^0(S^k\QQ\dual\otimes S^{k-1}\QQ)=0$
\end{description}
\end{lem}
\begin{proof}
See Lemma 2.2.18 of \cite{Toc}.
\end{proof}

This last lemma will be used in Theorem \ref{main theorem}:
\begin{lem}\label{cohom para main11}\label{cohom para main1111}\label{cohom para main111}\label{cohom para main1}\label{cohomologia para la demostracion}
Consider $\QQ$ the quotient bundle of rank $2$ over the Grassmannian of lines $\G(1,n)$ and $k\in\{1,\ldots,n-2\}$. Then,
\begin{description}
	\item[(i)] $H^{n-1}(S^k\QQ\otimes S^{n-k-1}\QQ(-n))=H^{n-1}(S^{n-1}\QQ(-n))=\Sc_{(1,\ldots,1)}\K^{n+1}=\bigwedge^{n+1}\K^{n+1}$ and its dimension is equal to $1$
	\item[(ii)] $H^{j}_*(S^k\QQ\otimes S^{n-k-1}\QQ)=0\quad\text{for}\quad j\neq n-1$
	\item[(iii)] $H^{j}_*(S^k\QQ\otimes S^{i}\QQ)=0\quad\text{for}\quad i\neq n-k-1 \quad\text{for}\quad j=1,2,\ldots,2n-3$
\end{description}
\end{lem}
\begin{proof}
See Lemma 2.2.19 of \cite{Toc}.
\end{proof}

\section{Splitting criterion}
In this section we recover the splitting criterion  given by E. Arrondo and F. Malaspina for the Grassmannian of lines in \cite{AM} which is a generalization of Theorem 2.1 of \cite{OTTcat} (given by G. Ottaviani) for the particular case of Grassmannian of lines. In \cite{OTTcat} the main tool is derived categories, but we use the Eagon-Northcott complexes and the Serre duality.
\begin{thm}\label{split theorem}
Suppose that for any coherent sheaf $\F$ of $\G(1,n)$ the following conditions are satisfied:
\begin{description}
	\item[i.] $H^0(\F(-1))=H^1(\F\otimes \QQ(-2))=H^2(\F\otimes S^2\QQ(-3))=\ldots=H^{n-2}(\F\otimes S^{n-2}\QQ(-n+1))=0$ 
	\item[ii.] $H^{n-1}(\F\otimes S^{n-2}\QQ(-n))=H^n(\F\otimes S^{n-3}\QQ(-n))=\ldots=H^{2n-5}(\F\otimes S^2\QQ(-n))=H^{2n-4}(\F\otimes \QQ(-n))=H^{2n-3}(\F(-n))=0$
\end{description}

Then $\F=(H^0(\F)\otimes \OO)\oplus \F'$.
\end{thm}
\begin{proof}
We prove the result by induction on the dimension of $H^0(\F)$, the case zero being trivial. We thus assume $H^0(\F)\neq 0$. We consider the Eagon-Northcott complex $(R_{n-1}\dual\otimes \OO(-1))$ glued together with $(R_{n-1})\otimes \OO(-n)$,
\begin{equation*}
0\longrightarrow \OO(-n-1)\longrightarrow \bigwedge^{n-1}V \otimes \OO(-n)\longrightarrow \bigwedge^{n-2}V\otimes\QQ(-n)\longrightarrow\ldots
\end{equation*}
\begin{equation*}
\ldots\longrightarrow \bigwedge^{2}V\otimes S^{n-3}\QQ(-n)\longrightarrow V\otimes S^{n-2}\QQ(-n)\longrightarrow V^*\otimes S^{n-2}\QQ(-n+1)\longrightarrow\ldots
\end{equation*}
\begin{equation}\label{complex}
\ldots \longrightarrow
	 	\bigwedge^{n-2}V^*\otimes\QQ(-2)\longrightarrow \bigwedge^{n-1}V^*\otimes\OO(-1)\longrightarrow \OO\longrightarrow 0
\end{equation}
which can be regarded as a generator of $H^{2n-2}(\omega_{\G(1,n)})=H^{2n-2}(\OO(-n-1))=Ext^{2n-2}(\OO,\OO(-n-1))$.
We can build the injective map $H^0(\F)\stackrel{\psi_1}{\longrightarrow} H^{n-1}(\F\otimes S^{n-1}\QQ(-n))$ by tensorizing the first part of the previous complex ($(R_{n-1}\dual\otimes \OO(-1))$) with $\F$ and using the vanishing in $(i.)$.
Using the second part of the previous complex ($(R_{n-1})\otimes \OO(-n)$) tensorized by $\F\dual$ and using the vanishing in $(ii.)$ one can build the surjective map $H^0(\F\dual)\stackrel{\psi_2}{\longrightarrow} H^{n-1}(\F\dual\otimes S^{n-1}\QQ\dual(-1))$. Hence, we can consider the following commutative diagram:

\resizebox{15.7cm}{!}{%
\begin{minipage}{\textwidth}
\centering
 $\begin{array}{rcl}
 		& &\\
	 H^{n-1}(\F\otimes S^{n-1}\QQ(-n))\times H^{n-1}(\F\dual\otimes S^{n-1}\QQ\dual(-1))  & \stackrel{\phi}{\longrightarrow} & H^{2n-2}(\OO(-n-1))\\
	 & &\\
	 \uparrow id\times\psi_2 & \circlearrowleft & \psi_4\uparrow\,\simeq \\
	 & & \\
	 H^{n-1}(\F\otimes S^{n-1}\QQ(-n))\times H^{0}(\F\dual) & \longrightarrow & H^{n-1}(S^{n-1}\QQ(-n)) \\
	 & &\\
	 \uparrow \psi_1\times id & \circlearrowleft & \psi_3\uparrow\,\simeq\\
	 & &\\
	 H^{0}(\F)\times H^{0}(\F\dual)  &\stackrel{\phi'}{\longrightarrow} & H^0(\OO) \\
	  & &
  \end{array}$
\end{minipage}}

One can easily see (with the help of Proposition \ref{producto tensorial de productos simetricos} and Proposition \ref{productos simetricos con cohomologia}) that the maps $\psi_3$ and $\psi_4$ are isomorhpisms.

Since the map $\phi$ is a perfect pairing, $\psi_1$ is injective, $\psi_2$ is surjective and the diagram commutates, a nonzero element $s\in H^0(\F)$ (which exists by hypothesis) has a partner $s^\vee\in H^0(\F^\vee)$ mapping to the identity in $H^0(\OO)$. This means that we can write $\F=\F'\oplus\OO$. Applying the induction hypothesis to $\F'$ we conclude the result.
\end{proof}

As a corollary for vector bundles we have the splitting criterion given by E. Arrondo and F. Malaspina for the Grassmannian of lines in \cite{AM}. Notice that we could have given a similar result as Theorem 2.2 of \cite{OTTcat} by adding the concept of g.skyscraper sheaf.
\begin{thm}\label{teorema de AM}
Let $F$ be a vector bundle on the Grassmannian of lines $\G(1,n)$. Then $F$ splits as a direct sum of line bundles if and only if:
\begin{description}
	\item[(a)] $H^1_*(F\otimes \QQ)=H^2_*(F\otimes S^2\QQ)=H^3_*(F\otimes S^3\QQ)=\ldots=H^{n-2}_*(F\otimes S^{n-2}\QQ)=0$ 
	\item[(b)] $H^{n-1}_*(F\otimes S^{n-2}\QQ)=H^n_*(F\otimes S^{n-3}\QQ)=\ldots=H^{2n-4}_*(F\otimes \QQ)=H^{2n-3}_*(F)=0$
\end{description}

\end{thm} 
\begin{proof}
Apply Theorem \ref{split theorem} using recursion on the rank of $F$.
\end{proof}

We can compare this characterization with the splitting criterion given by G. Ottaviani in \cite{OTT} for the particular case of the Grassmannian of lines $\G(1,n)$:
\begin{thm}\label{teorema ottaviani split} (G. Ottaviani, \cite{OTT})
A vector bundle $F$ over $\G(1,n)$ splits as a direct sum of line bundles if and only if the following conditions hold for $i=0,1,\ldots,n-2$:
$$H^j_*(F\otimes S^i\QQ)=0\quad\text{with}\quad i\leq j< 2n-2-i\quad\text{and}\quad j>0.$$
\end{thm}
In order to do the comparition  we introduce some graphical representation.

\subsection{Graphical representation}
We can express graphically the conditions of the form $H^j_*(F\otimes S^i\QQ)=0$ that appear in Theorem \ref{teorema de AM} as  the points $(i,j)$ in the following diagram (Figure \ref{fig:condiciones k=0}).
\begin{figure}[h!]
\begin{center}
\includegraphics[scale=0.5]{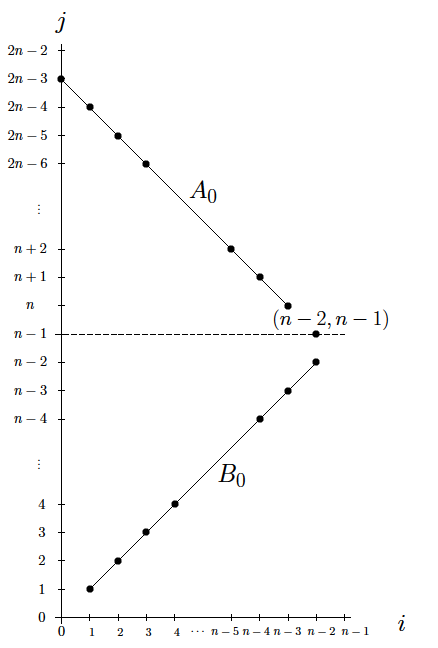}
\caption{\small Hypothesis for Theorem \ref{teorema de AM}}\label{fig:condiciones k=0}
\end{center}
\end{figure}
\newpage
Let us define the segments $A_0$ and $B_0$ appearing in the previous figure as follows:
\begin{itemize}
	\item $A_0=\{(0,2n-3),(1,2n-4),(2,2n-5),\ldots,(n-5,n+2),(n-4,n+1),(n-3,n)\}$
	\item $B_0=\{(1,1),(2,2),(3,3),\ldots,(n-4,n-4),(n-3,n-3),(n-2,n-2)\}$
\end{itemize}
We can compare graphically this splitting criterion with the one made by G. Ottaviani (Theorem \ref{teorema ottaviani split}). For this purpose let us define the following segments:
\begin{itemize}
	\item $L_0=\{(0,2n-3),(1,2n-4),\ldots,(n-3,n),(n-2,n-1)\}$
	\item For $k=1,\ldots,n-2$ we define $L_k=\{(0,2n-3-k),(1,2n-4-k),\ldots,(n-k-2,n-1),(n-k-1,n-2)\}=\{(j,2n-3-k-j)\quad\text{for}\quad j=0,1,\ldots,n-k-1\}$
	\item $R_0=\{(1,1),(2,2),\ldots,(n-3,n-3),(n-2,n-2)\}$
	\item For $k=1,2,\ldots,n-2$ we define $R_k=\{(0,k),(1,k+1),\ldots, (n-3-k,n-3),(n-2-k,n-2)\}=\{(j,k+j)\quad\text{for}\quad j=0,1,\ldots,n-k-2\}$
\end{itemize}
All the cohomological vanishings of Theorem \ref{teorema ottaviani split} consist in:
$$H^j_*(F\otimes S^i\QQ)=0\quad\text{where}\quad(i,j)\in
 \left\{
\begin{array}{l}
L_0\cup L_1\cup \ldots\cup L_{n-2}\\
R_0\cup R_1\cup\ldots\cup R_{n-2}
\end{array}\right.$$
These points correspond with Figure \ref{fig:splitting criteria de ottaviani}. Notice that $L_0=A_0\cup\{(n-2,n-1)\}$ and $R_0=B_0$.
\begin{figure}[h!]
\begin{center}
\includegraphics[scale=0.35]{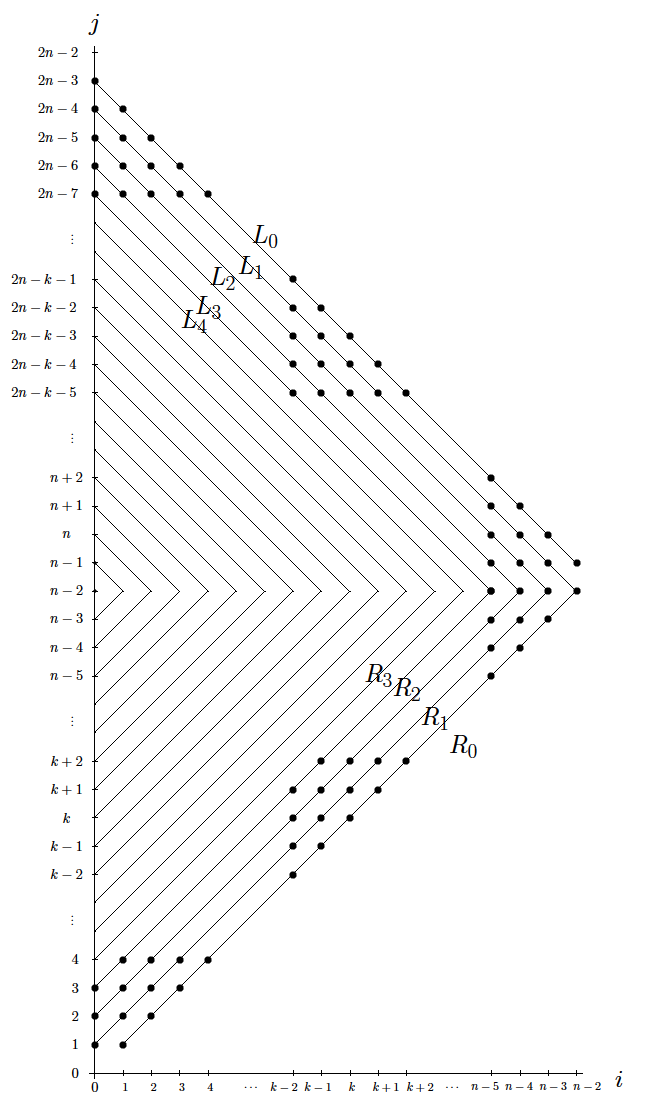}
\caption{\small Hypothesis for Theorem \ref{teorema ottaviani split}}\label{fig:splitting criteria de ottaviani}
\end{center}
\end{figure}
\newpage

\section{Main result}
We can now prove our characterization of direct sums of twists of symmetric powers of $\QQ$ over $\G(1,n)$.  We imitate the technique in \cite{AG}, having as starting point Theorem \ref{teorema de AM}. We first observe which of the hypotheses of this splitting criterion is not satisfied by $\QQ$ and try to see whether the remaining ones characterize vector bundles of the form $(\oplus\OO(l_{i_0}))\bigoplus(\oplus\QQ(l_{i_1}))$.
Observe that, by Lemma \ref{cohom para main1}, $\QQ$ satisfies all the conditions of this splitting criterion except one:
\begin{itemize}
	\item $H^j_*(\QQ\otimes S^j\QQ)=0$ \hspace{0.3cm} $j=1,2,\ldots,n-3,n-2$
	\item  $H^j_*(\QQ\otimes S^i\QQ)=0$ \hspace{0.3cm} $(i,j)\in \{(0,2n-3),(1,2n-4),\ldots,(n-3,n)\}$
	\item $H^{n-1}_*(\QQ\otimes S^{n-2}\QQ)\neq 0$
\end{itemize}

However, also $F=S^k\QQ(l)$ with $k\leq n-2$, satisfies all the hypotheses of Theorem \ref{teorema de AM} but $H^{n-1}_*(F\otimes S^{n-2}\QQ)=0$.
Therefore, we have to add more conditions. Once we have the characterization for the direct sums of twists of $\OO$, $\QQ$, $S^2\QQ$,\ldots, $S^{k-1}\QQ$ the idea is to remove one particular hypothesis and add a few more. In this way we get the characterization of direct sums of twists of $\OO$, $\QQ$,\ldots,$S^k\QQ$ with $k\leq n-2$.

\subsection{More graphical representations}
We continue with some graphical representation of the results that we will obtain.

\begin{remark}{\rm 
Notice that the point $(n-2,n-1)$ of Figure \ref{fig:condiciones k=0} is the continuation of the segment $A_0$ but we put it separately since it is the condition we need to remove to get the next characterization.  }
\end{remark}

\begin{remark}{\rm 
We have also expressed (in Figure \ref{fig:condiciones k=0}) with a dashed line the $n-1$ order of the cohomology of $H^j_*(F\otimes S^i\QQ)=0$ because all the cohomologies we remove in each step are in that line.   }
\end{remark}

We see in Theorem \ref{main theorem} that the characterization of $F=(\oplus\OO(l_{i_0}))\bigoplus(\oplus\QQ(l_{i_1}))$ consists of the vanishing of $H^j_*(F\otimes S^i\QQ)$ for $(i,j)$ in $A_0$, $B_0$ and the points that are shown in Figure \ref{fig:condiciones anadidas}.
\begin{figure}[h!]
\begin{center}
\includegraphics[scale=0.5]{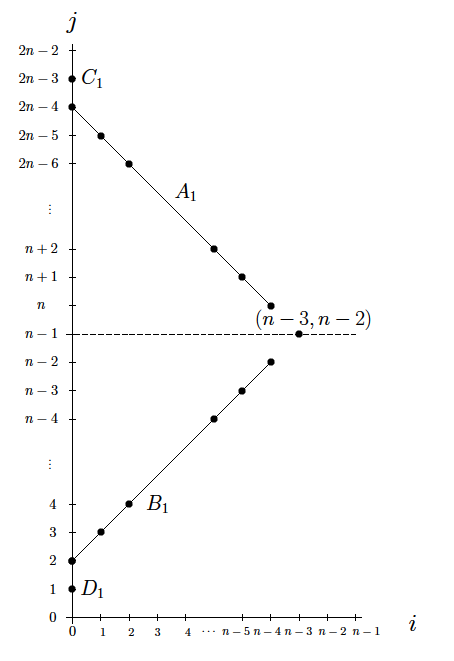}
\caption{\small Hypothesis to be added to characterize $F=(\oplus\OO(l_{i_0}))\bigoplus(\oplus\QQ(l_{i_1}))$}\label{fig:condiciones anadidas}
\end{center}
\end{figure}

Hence, if $F$ satisfies $H^j(F\otimes S^i\QQ)=0$ for $(i,j)$ as in conditions $A_0$ and $B_0$ of Figure \ref{fig:condiciones k=0} and as in all the conditions in Figure \ref{fig:condiciones anadidas} then we get the characterization for $F=(\oplus\OO(l_{i_0}))\bigoplus(\oplus\QQ(l_{i_1}))$. All these conditions can be expressed in Figure \ref{fig:todas condiciones k=1}.
\begin{figure}[h!]
\begin{center}
\includegraphics[scale=0.5]{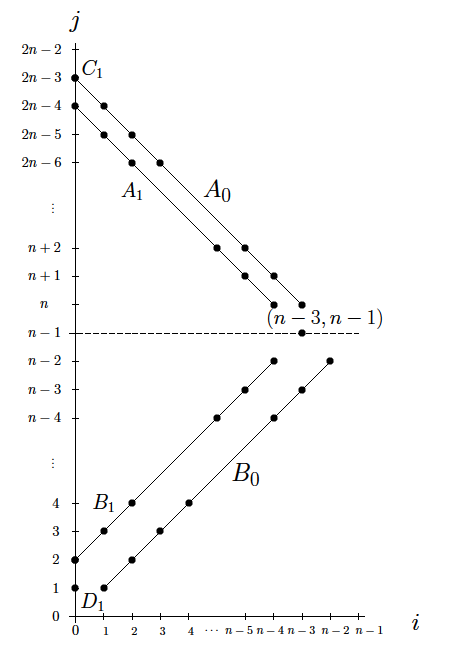}
\caption{\small Hypotheses needed to characterize $F=(\oplus\OO(l_{i_0}))\bigoplus(\oplus\QQ(l_{i_1}))$}\label{fig:todas condiciones k=1}
\end{center}
\end{figure}
\newpage
\begin{remark}{\rm
Notice that $C_1\subseteq A_0$ but we highlight it in Figure \ref{fig:condiciones anadidas} and \ref{fig:todas condiciones k=1} since in the iteration process we need the segments $C_k$ not contained in any other defined segment.   }
\end{remark}

In Theorem \ref{main theorem} we iterate this process to obtain a characterization of direct sums of twists of symmetric products of the universal bundle $$F=(\oplus\OO(l_{i_0}))\bigoplus(\oplus\QQ(l_{i_1}))\bigoplus\ldots\bigoplus (\oplus S^k\QQ(l_{i_k}))$$
until $k=n-2$.

For that, we give the definition of the segments $A_k$, $B_k$, $C_k$, $D_k$. Notice that we have already introduced $A_0$ and $B_0$ in the first part of this chapter.
\begin{defn}\label{definicion rectas}{\rm 
Let us define the following segments for $1\leq k\leq n-2$:
\begin{itemize}
	\item $A_k=\{(0,2n-k-3),(1,2n-k-4),(2,2n-k-5),\ldots,(n-k-4,n+1),(n-k-3,n)\}$
	\item $B_k=\{(0,k+1),(1,k+2),(2,k+3),\ldots,(n-k-4,n-3),(n-k-3,n-2)\}$
	\item $C_k=\{(0,2n-k-2),(1,2n-k-1),(2,2n-k),\ldots,(k-2,2n-4),(k-1,2n-3)\}$
	\item $D_k=\{(0,k),(1,k-1),(2,k-2),\ldots,(k-2,2),(k-1,1)\}$
\end{itemize}
Observe that $A_0$ can be defined in the same way as the $A_k$ with $k=0$ whereas $B_0$ cannot be defined as the $B_k$ with $k=0$.   }
\end{defn}

\begin{remark}\label{remark: para k=n-2 hay condiciones vacias}{\rm   
Notice that $A_{n-2}=\emptyset$ and $B_{n-2}=\emptyset$. Hence, in the case that $k\geq n-2$ we are not adding any cohomological condition with the notation of $A_k$ and $B_k$.

Moreover, all the conditions corresponding to the points of the lines $C_k$ and $D_k$ make no sense for $k\geq n-1$
because of the way they are defined. 
For example, if $k=n-1$ we get that $C_{n-1}=\{(0,n-1),(1,n),(2,n+1),\ldots,(n-3,2n-4),(n-2,2n-3)\}$ and $D_{n-1}=\{(0,n-1),(1,n-2),(2,n-3),\ldots,(n-3,2),(n-2,1)\}$. The point $(0,n-1)$ of $C_{n-1}$ and $D_{n-1}$ corresponds to the condition $H^{n-1}_*(F)=0$ and this cannot be true if we want to characterize $F=(\oplus\OO(l_{i_0}))\bigoplus(\oplus\QQ(l_{i_1}))\bigoplus\ldots\bigoplus(\oplus S^k\QQ(l_{i_k}))$ with $k=n-1$ since $H^{n-1}_*(S^{n-1}\QQ)\neq 0$ (see Proposition \ref{productos simetricos con cohomologia}).   }
\end{remark}

We show graphically in Figure \ref{fig:definicion de rectas} which are the points of the segments defined in Definition \ref{definicion rectas} (that correspond to the cohomological conditions $H^j_*(F\otimes S^i\QQ)=0$). We also draw the point $(n-k-2,n-1)$ since we use it in the lemmas and proposition of the following section.

\begin{figure}[h]
\begin{center}
\includegraphics[scale=0.34]{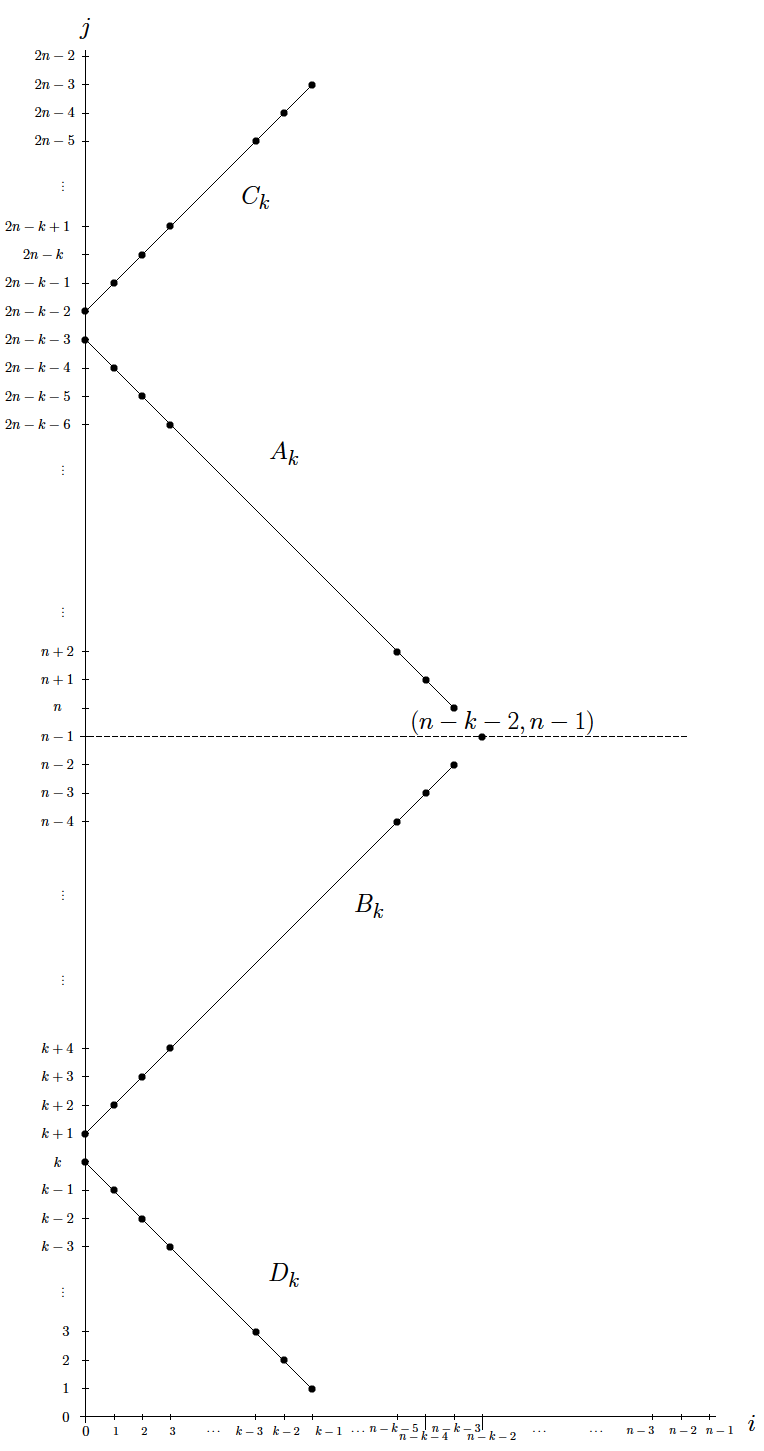}
\caption{\small $A_k$, $B_k$, $C_k$, $D_k$ and the point $(n-k-2,n-1)$}\label{fig:definicion de rectas}
\end{center}
\end{figure}

\begin{remark}{\rm
One can easily see that the points corresponding to:
$$(i,j)\in
 \left\{
\begin{array}{l}
A_0\cup A_1\cup \ldots\cup A_{k-1}\\
B_0\cup B_1\cup\ldots\cup B_{k-1}\\
C_1\cup C_2\cup\ldots\cup C_{k-1}\\
D_1\cup D_2\cup\ldots\cup D_{k-1}
\end{array}\right.$$
are inside of the set of points:
$$(i,j)\in
 \left\{
\begin{array}{l}
A_0\cup A_1\cup \ldots\cup A_{k}\\
B_0\cup B_1\cup\ldots\cup B_{k}\\
C_1\cup C_2\cup\ldots\cup C_{k}\\
D_1\cup D_2\cup\ldots\cup D_{k}
\end{array}\right.$$
 }
\end{remark}

\begin{figure}[h]
\begin{center}
\includegraphics[scale=0.39]{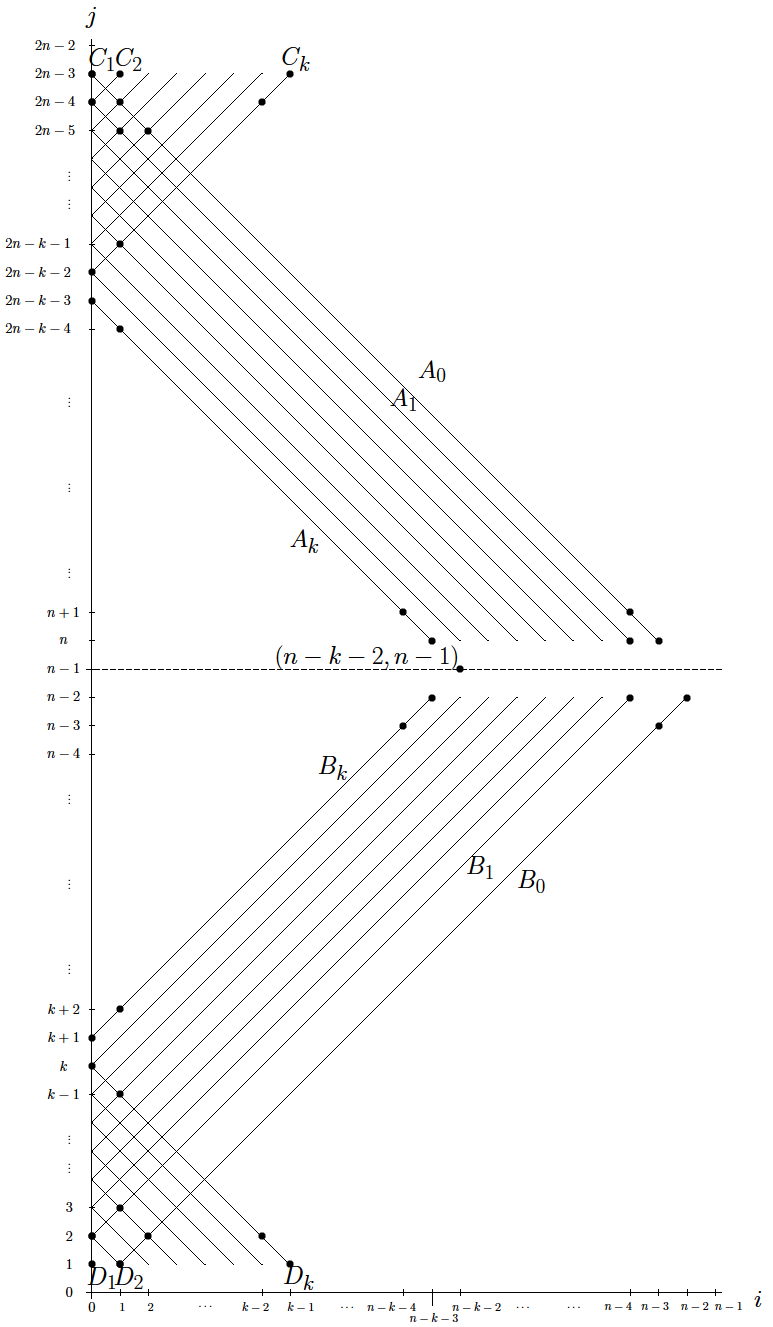}
\caption{\small Conditions to characterize direct sum of twists of $\OO, \QQ, S^2\QQ,\ldots, S^k\QQ$}\label{fig:caso k}
\end{center}
\end{figure}
This observation will be useful during the proof of the main theorem.
Specifically, these last conditions plus the point $(n-k-2,n-1)$ form the hypothesis of the main theorem. We show them graphically in Figure \ref{fig:caso k}. One starts drawing the lines corresponding to $A_k$, $B_k$, $C_k$, $D_k$ and fill the corresponding spaces with lines until the ones drawn in Figure \ref{fig:todas condiciones k=1} (except condition $(n-3,n-1)$).

\subsection{Previous lemmas}
Before giving the main result, let us give a general idea of 
the technique we use. To characterize 
when a vector bundle $F$ decomposes as
$(\oplus \OO(l_{i_0}))\bigoplus(\oplus\QQ(l_{i_1}))$ from the characterization of $\oplus \OO(l_{i_0})$ we can assume, after twisting $F$, $H^{n-1}(F\otimes S^{n-2}\QQ(-n))\neq 0$. 
Otherwise we could use Theorem \ref{teorema de AM}.
Then we use a particular Eagon-Northcott complex to transform $H^{n-1}(F\otimes S^{n-2}\QQ(-n))$ into $H^0(F\otimes \QQ\dual)$.
By Serre duality we have also $H^{n-1}(F\dual\otimes S^{n-2}\QQ\dual(-n-1))\neq 0$ and again, by a particular Eagon-Northcott complex, we transform it into $H^0(F\dual\otimes \QQ)$. Hence, we obtain the maps $\QQ\longrightarrow F$ and $F\longrightarrow \QQ$ and we
show that their composition is the identity of $\QQ$ or a multiple of it. Hence $\QQ$ is a direct summand of $F$ and we proceed by induction on $rank(F)$. 

Next we try to characterize when $F$ decomposes as $(\oplus\OO(l_{i_0}))\bigoplus(\oplus\ \QQ(l_{i_1}))\bigoplus(\oplus S^2\QQ(l_{i_2}))$. But this time we can assume $H^{n-1}(F\otimes S^{n-3}\QQ(-n))\neq 0$. By using Eagon-Northcott complexes and Serre duality we obtain the maps $S^2\QQ\longrightarrow F$ and $F\longrightarrow S^2\QQ$ and again we prove their composition to be the identity or a multiple of it, which allows to complete the proof.

In general we use induction on $k$ to characterize $(\oplus\OO(l_{i_0}))\bigoplus(\oplus\QQ(l_{i_1}))\bigoplus\ldots\bigoplus(\oplus\ S^k\QQ(l_{i_k}))$
from the characterization of $(\oplus\OO(l_{i_0}))\bigoplus(\oplus\QQ(l_{i_1}))\bigoplus\ldots\bigoplus(\oplus S^{k-1}\QQ(l_{i_{k-1}}))$ and for this purpose we assume $H^{n-1}(F\otimes S^{n-k-1}\QQ(-n))\neq 0$ (since $H^{n-1}(S^k\QQ\otimes S^{n-k-1}\QQ(-n)) \neq 0$ as we have observed in Lemma \ref{cohom para main111}). And again we get suitable maps $S^k\QQ\longrightarrow F$ and $F\longrightarrow S^k\QQ$. 

In the Lemmas of this section we see how to produce maps from the cohomology groups. We keep the notation of Definition \ref{definicion rectas}.

\begin{lem}\label{lema1}
Fix $k\in\{1,\ldots,n-2\}$. There exists a natural map
$$H^0(F\otimes S^k\QQ\dual)\stackrel{\psi_1}{\longrightarrow} H^{n-1}(F\otimes S^{n-k-1}\QQ(-n)).$$
Moreover, if the following conditions hold:
$$H^j_*(F\otimes S^i\QQ)=0\quad\text{with}\quad (i,j)\in B_k\cup D_k\cup (n-k-2,n-1)$$
the map $\psi_1$ is a surjective map.
\end{lem}
\begin{proof}
Since $k\leq n-2$ we consider the Eagon-Northcott complex  $(R\dual_{n-1-k})\otimes \OO(-k-1)$, glued together with $(R_k)\otimes \OO(-k)$ and using Remark \ref{isomorfismosentrefibradosuniversales} we get:
\begin{equation*}
0\longrightarrow S^{n-k-1}\QQ(-n)\longrightarrow V^*\otimes S^{n-k-2}\QQ(-n+1)\longrightarrow \bigwedge^2 V^*\otimes S^{n-k-3}\QQ(-n+2)\longrightarrow\ldots
\end{equation*}
\begin{equation*}
\ldots \longrightarrow \bigwedge^{n-k-2}V^*\otimes \QQ(-k-2)\longrightarrow \bigwedge^{n-k-1}V^*\otimes \OO(-k-1)\longrightarrow \bigwedge^k V\otimes \OO(-k)\longrightarrow \bigwedge^{k-1}V\otimes \QQ(-k)\longrightarrow \ldots
\end{equation*}
\begin{equation}\label{ea1}
\ldots\longrightarrow \bigwedge^2V\otimes S^{k-2}\QQ(-k)\longrightarrow V\otimes S^{k-1}\QQ(-k)\longrightarrow S^k\QQ(-k)\longrightarrow 0
\end{equation}

Now we tensorize it by $F$ and take all the short exact sequences
with their corresponding cokernels ($K\dual_{j}$). Then apply cohomology. Let us start from left to right:
\begin{equation*}
\resizebox{\textwidth}{!}{
\xymatrix{
H^{n-2}(F\otimes S^{n-k-1}\QQ(-n))\ar[r] & V^*\otimes H^{n-2}(F\otimes S^{n-k-2}\QQ(-n+1))\ar[r] & H^{n-2}(F\otimes K_{n-k-2}\dual(-k-1))\ar[dll]_{\varphi_1}\\
H^{n-1}(F\otimes S^{n-k-1}\QQ(-n)) \ar[r] &  V^*\otimes \underbrace{H^{n-1}(F\otimes S^{n-k-2}\QQ(-n+1))}_{\text{point } (n-k-2,n-1)}
 \ar[r] &  H^{n-1}(F\otimes K_{n-k-2}\dual(-k-1))
}}
\end{equation*}

Moreover, if $H^{n-1}_*(F\otimes S^{n-k-2}\QQ)=0$, then $\varphi_1$ is a surjective map. Let us repeat the same argument.
By using the first part of (\ref{ea1}) we construct $\varphi_j$ for $j=2,\ldots,n-k-2$:
\begin{equation*}
\resizebox{\textwidth}{!}{
\xymatrix{
 H^{n-j-1}(F\otimes K_{n-k-j}\dual(-k-1))\ar[r] &\bigwedge^j V^*\otimes H^{n-j-1}(F\otimes S^{n-k-j-1}\QQ(-n+j))\ar[r] & H^{n-j-1}(F\otimes K_{n-k-j+1}\dual(-k-1))\ar[dll]_{\varphi_{j}}\\
 H^{n-j}(F\otimes K_{n-k-j}\dual(-k-1)) \ar[r] &  \bigwedge^j V^*\otimes \underbrace{H^{n-j}(F\otimes S^{n-k-j-1}\QQ(-n+j))}_{\text{point } (n-k-j-1,n-j)\in B_k}
 \ar[r] &  H^{n-j}(F\otimes K_{n-k-j+1}\dual(-k-1))
}}
\end{equation*}
We use the middle part of (\ref{ea1}) for the maps $\varphi_{n-k-1}$ and $\varphi_{n-k}$:
\begin{equation*}
\resizebox{\textwidth}{!}{
\xymatrix{
H^{k}(F\otimes K_{1}\dual(-k-1))\ar[r] & \bigwedge^{n-k-1} V^*\otimes  H^{k}(F(-k-1))\ar[r] & H^{k}(F\otimes \bigwedge^{n-k-1}\s(-k-1))\ar[dll]_{\varphi_{n-k-1}}\\
H^{k+1}(F\otimes K_{1}\dual(-k-1)) \ar[r] & \bigwedge^{n-k-1} V^*\otimes \underbrace{ H^{k+1}(F(-k-1))}_{\text{point } (0,k+1)\in B_k}
 \ar[r] &   H^{k+1}(F\otimes \bigwedge^{n-k-1}\s(-k-1))
}}
\end{equation*}
\begin{equation*}
\resizebox{\textwidth}{!}{
\xymatrix{
 H^{k-1}(F\otimes\bigwedge^k\s\dual(-k))\ar[r] & \bigwedge^k V\otimes H^{k-1}( F(-k))\ar[r] &  H^{k-1}(F\otimes K'_{1}(-k))\ar[dll]_{\varphi_{n-k}}\\
 H^{k}(F\otimes\bigwedge^k\s\dual(-k)) \ar[r] &   \bigwedge^k V\otimes\underbrace{H^{k}( F(-k))}_{\text{point } (0,k)\in D_k}
 \ar[r] &   H^{k}(F\otimes \bigwedge^{n-k-1}\s(-k-1))
}}
\end{equation*}
We repeat now for $\varphi_j$ with $j=n-k+1,\ldots,n-2$ by using the third part of (\ref{ea1}), and finally for $\varphi_{n-1}$:
\begin{equation*}
\resizebox{\textwidth}{!}{
\xymatrix{
H^{n-j-1}(F\otimes K'_{k+j-n}(-k))\ar[r] & \bigwedge^{n-j} V\otimes H^{n-j-1}(F\otimes S^{k+j-n}\QQ(-k))\ar[r] &  H^{n-j-1}(F\otimes K'_{k+j-n+1}(-k))\ar[dll]_{\varphi_{j}}\\
 H^{n-j}(F\otimes K'_{k+j-n}(-k)) \ar[r] &   \bigwedge^{n-j} V\otimes\underbrace{H^{n-j}(F\otimes S^{k+j-n}\QQ(-k))}_{\text{point } (k+j-n,n-j)\in D_k}
 \ar[r] &  H^{n-j}(F\otimes K'_{k+j-n+1}(-k))
}}
\end{equation*}
\begin{equation*}
\resizebox{14cm}{!}{
\xymatrix{
H^{0}(F\otimes K'_{k-1}(-k))\ar[r] & V\otimes H^{0}( F\otimes S^{k-1}\QQ(-k))\ar[r] &  H^{0}(F\otimes S^k\QQ(-k))\ar[dll]_{\varphi_{n-1}}\\
H^{1}(F\otimes K'_{k-1}(-k)) \ar[r] &   V\otimes\underbrace{H^{1}( F\otimes S^{k-1}\QQ(-k))}_{\text{point } (k-1,1)\in D_k}
 \ar[r] &   H^{1}(F\otimes S^k\QQ(-k))
}}
\end{equation*}

By the composition of all $\varphi_j$ for $j=1,\ldots,n-1$ we have finally constructed the map $\psi_1$:
$$H^{0}(F\otimes S^{k}\QQ(-k))\stackrel{\psi_1}{\longrightarrow} H^{n-1}(F\otimes S^{n-k-1}\QQ(-n))$$

Moreover, if all the cohomologies of the statement vanish then $\psi_1$ is a surjective map.
\end{proof}

\begin{remark}{\rm
Notice that the cohomological condition corresponding
to the point $(n-k-2,n-1)$ makes no sense if $k> n-2$. This is why Theorem \ref{main theorem} will be valid only for $k\leq n-2$.   }
\end{remark}

For the next lemma we need the conditions of the form
$H^j_*(F\dual\otimes S^i\QQ)=0$ for some $(i,j)$. More precisely, the conditions we need are:
\begin{itemize}
	\item $H^{n-2}_*(F\dual\otimes S^{n-k-3}\QQ\dual)=H^{n-3}_*(F\dual\otimes S^{n-k-4}\QQ\dual)=\ldots$\\
	$\ldots=H^{k+3}_*(F\dual\otimes S^{2}\QQ\dual)=H^{k+2}_*(F\dual\otimes \QQ\dual)=H^{k+1}_*(F\dual)=0$
	\item $H^{k}_*(F\dual)=H^{k-1}_*(F\dual\otimes\QQ\dual)=H^{k-2}_*(F\dual\otimes S^{2}\QQ\dual)=\ldots$\\
	$\ldots=H^{3}_*(F\dual\otimes S^{k-3}\QQ)=H^{2}_*(F\dual\otimes S^{k-2}\QQ)=H^{1}_*(F\dual\otimes S^{k-1}\QQ)=0$
	\item $H^{n-1}_*(F\dual\otimes S^{n-k-2}\QQ\dual)=0$
\end{itemize}
For simplicity, by using the Serre duality we transform them into another ones without $F\dual$:
\begin{itemize}
	\item $H^{n}_*(F\otimes S^{n-k-3}\QQ)=H^{n+1}_*(F\otimes S^{n-k-4}\QQ)=\ldots$\\
	$\ldots=H^{2n-5-k}_*(F\otimes S^{2}\QQ)=H^{2n-4-k}_*(F\otimes \QQ)=H^{2n-3-k}_*(F)=0$
	\item $H^{2n-2-k}_*(F)=H^{2n-1-k}_*(F\otimes \QQ)=H^{2n-k}_*(F\otimes S^{2}\QQ)=\ldots$\\
	$\ldots=H^{2n-5}_*(F\otimes S^{k-3}\QQ)=H^{2n-4}_*(F\otimes S^{k-2}\QQ)=H^{2n-3}_*(F\otimes S^{k-1}\QQ)=0$
	\item $H^{n-1}_*(F\otimes S^{n-k-2}\QQ)=0$
\end{itemize}
These correspond to the segments $A_k$, $C_k$ and the point $(n-k-2,n-1)$.

\begin{lem}\label{lema2}
Fix $k\in\{1,\ldots,n-2\}$. There exists a natural map:
$$H^0(F\dual\otimes S^k\QQ)\stackrel{\psi_2}{\longrightarrow} H^{n-1}(F\dual\otimes S^{n-k-1}\QQ\dual(-1))$$
Moreover, if the following conditions hold:
$$H^j_*(F\otimes S^i\QQ)=0\quad\text{with}\quad (i,j)\in A_k\cup C_k\cup (n-k-2,n-1)$$
the map $\psi_2$ is a surjective map.
\end{lem}

\begin{proof}
Argue in the same way as before. Since $k\leq n-2$ we can consider the Eagon-Northcott complex  $(R\dual_{n-1-k})\otimes \OO(-1)$ glue together with $(R_k)$:
\begin{equation*}
0\longrightarrow S^{n-k-1}\QQ\dual(-1)\longrightarrow V^*\otimes S^{n-k-2}\QQ\dual(-1)\longrightarrow \bigwedge^2 V^*\otimes S^{n-k-3}\QQ\dual(-1)\longrightarrow\ldots
\end{equation*}
\begin{equation*}
\ldots \longrightarrow \bigwedge^{n-k-2}V^*\otimes \QQ\dual(-1)\longrightarrow \bigwedge^{n-k-1}V^*\otimes \OO(-1)\longrightarrow \bigwedge^k V\otimes \OO\longrightarrow \bigwedge^{k-1}V\otimes \QQ\longrightarrow \ldots
\end{equation*}
\begin{equation}\label{ea2}
\ldots\longrightarrow \bigwedge^2V\otimes S^{k-2}\QQ\longrightarrow V\otimes S^{k-1}\QQ\longrightarrow S^k\QQ\longrightarrow 0
\end{equation}
This time we tensorize it by $F\dual$ and take all the short exact sequences
with their corresponding cokernels ($K\dual_{j}$) as before.
\end{proof}

\begin{notation}{\rm
We want to remark the following isomorphism:
$$\K=H^{2n-2}(\OO(-n-1))\cong H^{2n-2}(S^{n-k-1}\QQ(-n)\otimes S^{n-k-1}\QQ\dual(-1)).$$
We call the following perfect pairing:
\begin{equation*}
 \begin{array}{rcl}
	 H^{n-1}(F\otimes S^{n-k-1}\QQ(-n))\times H^{n-1}(F\dual\otimes
	  S^{n-k-1}\QQ\dual(-1)) & \stackrel{\phi}{\longrightarrow} & H^{2n-2}(\OO(-n-1)) \\
	  (\alpha,\beta) & \longmapsto & \phi((\alpha,\beta)) 
  \end{array}
\end{equation*}	
the \textit{Serre pairing}.   }
\end{notation}

Now we are ready to state the following proposition in which we prove that the composition of the maps $S^k\QQ\longrightarrow F$ and $F\longrightarrow S^k\QQ$ corresponds to the identity or a multiple of the identity.

\begin{prop}\label{prop}
Fix $k\in\{1,\ldots,n-2\}$ and consider $\alpha$ and $\beta$ non zero elements such that the image of the Serre pairing:
\begin{equation*}
 \begin{array}{rcl}
	 H^{n-1}(F\otimes S^{n-k-1}\QQ(-n))\times H^{n-1}(F\dual\otimes
	  S^{n-k-1}\QQ\dual(-1)) & \stackrel{\phi}{\longrightarrow} & H^{2n-2}(\OO(-n-1)) \\
	  (\alpha,\beta) & \longmapsto & \phi((\alpha,\beta)) 
  \end{array}
\end{equation*}	
is non zero. Suppose that the following conditions hold:
$$H^j_*(F\otimes S^i\QQ)=0\quad\text{with}\quad (i,j)\in A_k\cup B_k\cup C_k\cup D_k\cup (n-k-2,n-1)$$
Then we can lift $\alpha$ to $\alpha'\in H^0(F\otimes S^{k}\QQ\dual)$ and $\beta$ to $\beta'\in H^0(F\dual\otimes S^{k}\QQ)$ by the natural maps,
\begin{equation*}
 \begin{array}{rcl}
	 H^0(F\otimes S^{k}\QQ\dual) & \stackrel{\psi_1}{\longrightarrow} & H^{n-1}(F\otimes S^{n-k-1}\QQ(-n)) \\
	  \alpha' & \longmapsto & \alpha 
  \end{array}
\end{equation*}	
\begin{equation*}
 \begin{array}{rcl}
	 H^0(F\dual\otimes S^{k}\QQ) & \stackrel{\psi_2}{\longrightarrow} & H^{n-1}(F\dual\otimes S^{n-k-1}\QQ\dual(-1)) \\
	  \beta' & \longmapsto & \beta 
  \end{array}
\end{equation*}	
of Lemma \ref{lema1} and Lemma \ref{lema2},
such that, regarding $\alpha'\in Hom(S^k\QQ, F)$ and $\beta'\in Hom(F,S^k\QQ)$, their composition is a nonzero multiple of the identity map $S^k\QQ\longrightarrow S^k\QQ$.
\end{prop}
\begin{proof}
By Lemma \ref{lema1} and Lemma \ref{lema2} we can lift $\alpha$ to $\alpha'$ by $\psi_1$ and by $\psi_2$ we can lift $\beta$ to $\beta'$. We would like also to lift $\phi$ to the composition map $Hom(S^k\QQ,F)\times Hom(F,S^k\QQ)\longrightarrow Hom(S^k\QQ,S^k\QQ)$. So, let us identify $H^{2n-2}(\OO(-n-1))$ with $Hom(S^k\QQ,S^k\QQ)$. We do this
by using the same complexes we used for the maps $\psi_1$ and $\psi_2$ since $k\leq n-2$.

If we tensorize now the complex (\ref{ea1}) by  $S^{n-k-1}\QQ\dual(-1)$ instead of by $F$ and repeat the same argument as in Lemma \ref{lema1}, but starting from $2n-2$ as the higher order of the cohomology, we obtain that the following map is an isomorphism:
$$H^{n-1}(S^k\QQ\dual\otimes S^{n-k-1}\QQ\dual(-1)) \xrightarrow[\psi_3]{\simeq}H^{2n-2}(S^{n-k-1}\QQ(-n)\otimes S^{n-k-1}\QQ\dual(-1))$$
The cohomology vanishings we need are the ones given in Lemma \ref{cohom para lema 1}.

Now let us consider the complex (\ref{ea2}) tensorized by $S^k\QQ\dual$
instead of by $F\dual$ from Lemma \ref{lema2}. This time we start from the cohomology of order $n-1$ and obtain that the following map is an isomorphism:
$$H^{0}(S^k\QQ\dual\otimes S^{k}\QQ)\xrightarrow[\psi_4]{\simeq} H^{n-1}(S^k\QQ\dual\otimes S^{n-k-1}\QQ\dual(-1))$$
As before, the cohomology vanishings we need are the ones given in Lemma \ref{cohom para lema 3}.

Putting all together we get the following commutative diagram:
$$H^{n-1}(F\otimes S^{n-k-1}\QQ(-n))\times H^{n-1}(F\dual\otimes S^{n-k-1}\QQ\dual(-1))  \stackrel{\phi}{\longrightarrow} H^{2n-2}(S^{n-k-1}\QQ(-n)\otimes S^{n-k-1}\QQ\dual(-1))$$
$$\uparrow\psi_1\times id\qquad\qquad\qquad\qquad\qquad\qquad\qquad\circlearrowleft\qquad\psi_3\uparrow\,\simeq\qquad\qquad$$
$$\qquad H^{0}(F\otimes S^{k}\QQ\dual)\times H^{n-1}(F\dual\otimes S^{n-k-1}\QQ\dual(-1))  \longrightarrow H^{n-1}(S^k\QQ\dual\otimes S^{n-k-1}\QQ\dual(-1))$$
$$\qquad\qquad\qquad id\times\psi_2\uparrow\qquad\qquad\circlearrowleft\qquad\psi_4\uparrow\,\simeq$$
$$ \qquad H^{0}(F\otimes S^{k}\QQ\dual)\times H^{0}(F\dual\otimes S^{k}\QQ)  \stackrel{\phi'}{\longrightarrow} H^0(S^k\QQ\dual\otimes S^k\QQ)$$
where $\psi_1$ and $\psi_2$ are surjective maps and $\psi_3$ and $\psi_4$ are isomorphisms.

Hence, the  composition of the morphisms $\alpha'\in Hom(S^k\QQ,F)$ and $\beta'\in Hom(F\dual\otimes S^k\QQ)$ is a non zero map $S^k\QQ\longrightarrow S^k\QQ$. Since $S^k\QQ$ is simple (see Proposition \ref{es simple}) this composition is necessarily a multiple of the identity. 
\end{proof}

\begin{remark}{\rm
We can identify the hypothesis of Proposition \ref{prop} with the points $(i,j)$ in Figure \ref{fig:definicion de rectas}.   }
\end{remark}

We prove now the main result.
\begin{thm}\label{main theorem}
Fix  $k\in\{0,1,\ldots,n-2\}$ and let $F$ be a vector bundle over the Grassmannian of lines $\G(1,n)$.
Then $F$ is a direct sum of twists of $\OO, \QQ, S^2\QQ,\ldots, S^k\QQ$
 if and only if the following conditions hold:
 
 $$H^j_*(F\otimes S^i\QQ)=0\quad\text{for}\quad (i,j)\in
 \left\{
\begin{array}{l}
A_0\cup A_1\cup \ldots\cup A_k\\
B_0\cup B_1\cup\ldots\cup B_k\\
C_1\cup C_2\cup\ldots\cup C_k\\
D_1\cup D_2\cup\ldots\cup D_k\\
(n-k-2,n-1)
\end{array}\right.$$
\end{thm}

\begin{proof}
We use double induction. First we make induction on $k$, the case $k=0$ being Theorem \ref{teorema de AM}.
We suppose now the theorem true for $k-1$ and we want to prove it for $k$.
Since $k\leq n-2$ we can now apply induction on $m:=\sum_l h^{n-1} (F\otimes S^{n-k-1}\QQ(l))$.
When $m=0$ we are in the hypothesis of the theorem when replacing $k$ with $k-1$. Hence by induction hypothesis, $F$ can be expressed as the direct sum of twist of $\OO,\,\QQ,\,S^2\QQ,\,\ldots,S^{k-1}\QQ$.

Assume now $m>0$ and that we know the result for $m-1$. In particular, $H^{n-1}(F\otimes S^{n-k-1}\QQ(l))\neq 0$ for some $l$. We want to show that  $S^k\QQ(-n-l)$ is a direct summand of $F$. 
Since the result is independent on the twist, we can assume $l=-n$. We take a non zero element $\alpha\in H^{n-1}(F\otimes S^{n-k-1}\QQ(-n))$. Since
 the Serre pairing:
\begin{equation*}
 \begin{array}{rcl}
	 H^{n-1}(F\otimes S^{n-k-1}\QQ(-n))\times H^{n-1}(F\dual\otimes
	  S^{n-k-1}\QQ\dual(-1)) & \stackrel{\phi}{\longrightarrow} & H^{2n-2}(\OO(-n-1)) 
  \end{array}
\end{equation*}	
is perfect we can take $\beta \in  H^{n-1}(F\dual\otimes
S^{n-k-1}\QQ\dual(-1))$ such that $\phi(\alpha,\beta)\neq 0$.

Since $k\leq n-2$, by Proposition \ref{prop} we can lift $\alpha$ to $\alpha'\in Hom(S^k\QQ)$ and $\beta$ to $\beta'\in Hom(F,S^k\QQ)$
such that the composition of $\alpha'$
and $\beta'$ is a nonzero multiple of the identity map $S^k\QQ\longrightarrow S^k\QQ$.
Thus we can write $F=S^k\QQ\oplus F'$ for some $F'$. 
Clearly $F'$ satisfies all the hypothesis of the theorem (since $F$ does). Moreover, since $h^{n-1}(S^k\QQ\otimes S^{n-k-1}\QQ(-n))=1$ by Lemma \ref{cohomologia para la demostracion}  we have $h^{n-1}(F'\otimes S^{n-k-1}\QQ(-n))=m-1$.

By induction hypothesis $F'$ can be expressed as direct sums of twists of $\OO,\QQ,S^2\QQ,\ldots,S^k\QQ$ and hence the same holds for $F$ as we wanted.
\end{proof}

\section{Final remarks}
\subsection{Comparison with derived categories}
As A. Kuznetsov suggested, a natural way to obtain cohomological characterization of vector bundles is the use of derived categories (see \cite{Kuz}). The main tool that uses this technique is the resolution of the diagonal. For the case of $\G(k,n)$ this resolution is:
\begin{equation}\label{resolucion kapranov}
0\rightarrow\bigwedge^{(k+1)(n-k)}(\QQ\dual\boxtimes\s\dual)\rightarrow\ldots \bigwedge^2(\QQ\dual\boxtimes\s\dual)\rightarrow \QQ\dual\boxtimes\s\dual\rightarrow\OO_{X\times X}\rightarrow\OO_\Delta\rightarrow 0
\end{equation}

\begin{remark}\label{ultimo remark}{\rm
The main difference now is that the wedge products of $\QQ\dual\boxtimes\s\dual$ are much more complicated and decompose into many pieces. Specifically:
$$\bigwedge^k(\QQ\dual\boxtimes\s\dual)=\bigoplus_{|\lambda|=k} \Sc_\lambda\QQ\dual\otimes\Sc_{\lambda'}\s\dual$$ where the sum goes over all Young tableau with $k$ cells, $\lambda'$ is the conjugate Young tableau and $\Sc_\lambda$ is the Schur functor associated to the tableau $\lambda$.  }
\end{remark}

By applying Beilinson's theorem we obtain the same characterization given in Theorem \ref{split theorem} and Theorem \ref{teorema de AM} but with different conditions. In order to state these statements we must give some definitions. For more details see Section 4.3 of \cite{Toc}.

\begin{defn} {\rm
Consider the following set of points:
\begin{itemize}
	\item $M_0=\{(0,2n-3),(1,2n-4),(2,2n-5),(3,2n-6),\ldots,(n-5,n+2),(n-4,n+1),(n-3,n)\}$
	\item $M'=\{(0,2n-4),(1,2n-5),(2,2n-6),(3,2n-7),\ldots, (n-4,n),(n-3,n-1),(n-2,n-2)\}$
	\item $M_k=\{(0,2n-3-2k),(1,(2n-3-2k)-1),(2,(2n-3-2k)-2),(3,(2n-3-2k)-3),\ldots, (n-2-2k,n-1),(n-1-2k,n-2)\}$
	\item $N_0=\{(2,1),(3,2),(4,3),(5,4),\ldots, (n-4,n-5),(n-3,n-4),(n-2,n-3)\}$
	\item $N_k=\{(0,1+2(k-1)),(1,(1+2(k-1))+1),(2,(1+2(k-1))+2),(3,(1+2(k-1))+3),\ldots,(n-4-2(k-1),n-3),(n-3-2(k-1),n-2)\}$
\end{itemize}
for $k\in\{1,2,\ldots,[\frac{n-1}{2}]\}$.   }
\end{defn}

\begin{thm}\label{split}
Let $\F$ be a coherent sheaf over $\G(1,n)$. Suppose that for some $t\in \Z$ we have the following vanishings:
$$H^j(\F\otimes S^i\QQ\dual(t-\frac{j-i+1}{2}))=0\quad\text{for}\quad (i,j)\in
 \left\{
\begin{array}{l}
(1,0)\\
M_0\cup M'\cup M_1\cup \ldots\cup M_{[\frac{n-1}{2}]}\\
N_0\cup N_1\cup\ldots\cup N_{[\frac{n-1}{2}]}
\end{array}\right.$$
Then $\F$ contains $\OO(t)\otimes H^0(\F(t))$
as a direct summand.
\end{thm}
\begin{proof}
See Theorem 4.3.6 of \cite{Toc}.
\end{proof}

Hence, we can state the splitting criterion for vector bundles over $\G(1,n)$ (notice again that we could have given the result for coherent sheaf by adding the concept of g.skyscraper sheaf).
\begin{thm}\label{splitting criteria final}
Let $F$ be a vector bundle over $\G(1,n)$. Suppose that:
$$H^j_*(F\otimes S^i\QQ)=0\quad\text{for}\quad (i,j)\in
 \left\{
\begin{array}{l}
M_0\cup M'\cup M_1\cup \ldots\cup M_{[\frac{n-1}{2}]}\\
N_0\cup N_1\cup\ldots\cup N_{[\frac{n-1}{2}]}
\end{array}\right.$$
Then $F$ is a direct sum of line bundles.
\end{thm}
\begin{proof}
See Theorem 4.3.9 of \cite{Toc}.
\end{proof}

Let us show graphically the conditions of Theorem \ref{splitting criteria final} in Figure \ref{fig:splitting categorias II}.

\begin{figure}[h!]
\begin{center}
\includegraphics[scale=0.52]{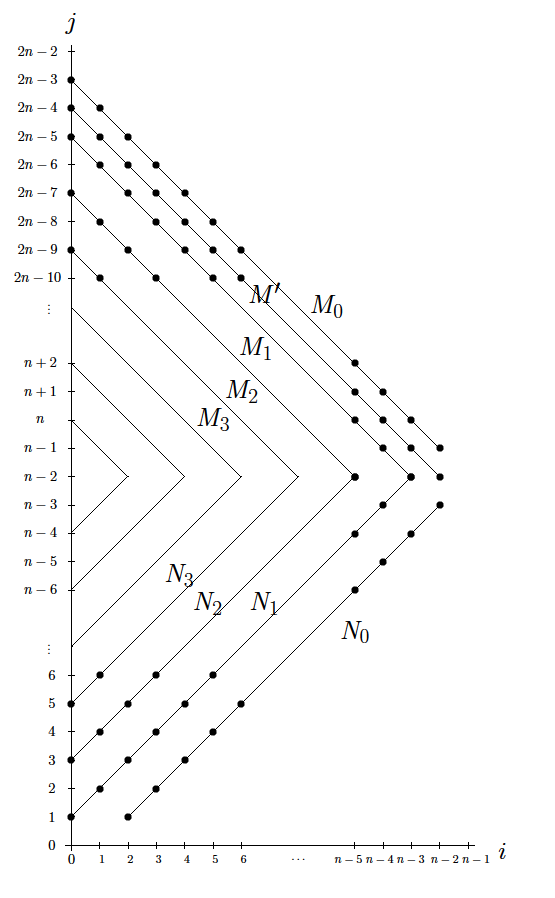}
\caption{\small Hypotheses of Theorem \ref{splitting criteria final}}\label{fig:splitting categorias II}
\end{center}
\end{figure}
\newpage

Notice that for $(i,j)=(1,0)$ the corresponding condition is $H^0(F\otimes \QQ\dual(t))=0$. This condition plays an important role since it does not appear in the conditions of Theorem \ref{splitting criteria final} but one can prove it holds with all the conditions we have in the hypotheses.

\begin{remark}{\rm
As one can observe, the splitting criterion of Theorem \ref{splitting criteria final} implies the criterion by Ottaviani (for more details see Remark 4.3.10 of \cite{Toc}). 

We can also compare graphically Theorem \ref{splitting criteria final} with the splitting criterion given by E. Arrondo and F. Malaspina (Theorem \ref{teorema de AM}).
We see that Theorem \ref{teorema de AM} only needs $2n-2$ conditions, while Theorem \ref{splitting criteria final} requires much more conditions. The reason is implicit in Remark \ref{ultimo remark}, because each of the $2n-2$ terms of the resolution of the diagonal is made of many different pieces, each of them imposing a different condition.   }
\end{remark}

\subsection{Searching a Serre extension for $\G(k,n)$}
Now we want to generalize the argument used to prove the main theorem for a general Grassmannian $\G(k,n)$. Everything works properly for the particular case of Grassmannian of lines $\G(1,n)$ since we have a nice extension (\ref{complex}). Each element of this extension does not have intermediate cohomology. Moreover, it is a nonzero element of $H^{2n-2}(\OO(-n-1))$, i.e., an element of $Ext^{2n-2}(\OO,\OO(-n-1))$ with dimension $1$ by Serre duality. This is what we call \textit{Serre extension}.

In order to give a Serre extension for $\G(k,n)$ it seems that one can use the sequences in the following proposition given by A. Fonarev in \cite{For} which are some kind of generalization of the Eagon-Northcott complex. This proposition gives the corresponding extension $Ext^{n-k}(\Sc_\lambda \QQ,\Sc_{\lambda^*}\QQ(-1))$ for a particular Young diagram $\lambda$ with $\lambda_1=n-k$.

\begin{prop}\label{generalizacion eagon} {\rm (Proposition 5.3 of \cite{For})}\
Let $\lambda\in Y_{n,k}$  be a diagram with $\lambda=(n-k,\lambda_2,\ldots,\lambda_{k+1})$ (where $Y_{n,k}$ is the set of Young diagrams inscribed in a rectangle of size $(k+1)\times(n-k)$). Consider $\lambda^*=(\lambda_2,\ldots,\lambda_{k+1})$. Then there exists the following long exact sequence for the tautological subbundle $\QQ\dual$ of rank $k+1$,
\begin{equation*}
0\longrightarrow \Sc_{\lambda^*}\QQ(-1)\longrightarrow \bigwedge^{\upsilon_{n-k}} V\otimes \Sc_{\mu_{n-k}}\QQ\longrightarrow \bigwedge^{\upsilon_{n-k-1}} V\otimes \Sc_{\mu_{n-k-1}}\QQ\longrightarrow\ldots
\end{equation*}
\begin{equation}\label{sucesion schur}
\ldots\longrightarrow
\bigwedge^{\upsilon_2} V\otimes\Sc_{\mu_2}\QQ\longrightarrow
\bigwedge^{\upsilon_1}V\otimes \Sc_{\mu_1}\QQ\longrightarrow
\Sc_\lambda\QQ\longrightarrow 0
\end{equation}
for some $0\leq \upsilon_i<n$ and $\mu_i\in Y_{n,k}$ where $\upsilon_i$ and $\mu_i$ are defined in a nice combinatorial way. 
Given $\lambda\in Y_{n,k}$ with $\lambda_1=n-k$, let us  draw a strip of width $1$ as shown:

\begin{figure}[h!]
\begin{center}
\includegraphics[scale=0.6]{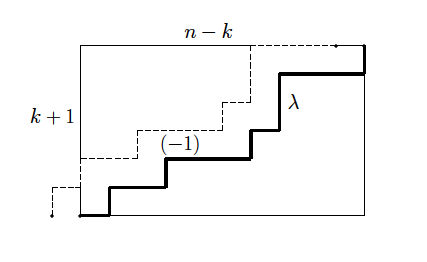}
\end{center}
\end{figure}
\newpage
Now, $\mu_i$ is pictured in the following figure by the solid line. One takes the path $\lambda$ going from left to right and
"jumps" upward on the path $\lambda'(-1)$ in the point with abscissa $n-k-i$.
\begin{figure}[h!]
\begin{center}
\includegraphics[scale=0.6]{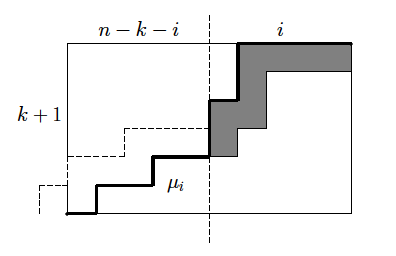}
\end{center}
\end{figure}
The number $\upsilon_i$
is the number of boxes one needs to remove from $\lambda$ in order to get $\mu_i$. These are pictured
in gray on the previous figure.
\end{prop}

Some of the vector bundles appearing in (\ref{sucesion schur}) have intermediate cohomology. However, by hand and with the help of Schubert package, one can construct a similar sequence but with better properties, since all the elements of the sequence will not have intermediate cohomology and there will be some duality conditions too.   
We give as an example the case of $\G(2,5)$.

\begin{example}\label{mejoraFonarev}{\rm 
For $\G(2,5)$ one can build the following extension by using (\ref{sucesion schur}):
$$0\rightarrow \OO(-3)\rightarrow V^*\otimes \Sc_{(3,3,2)}\QQ\dual\rightarrow\bigwedge^2 V^*\otimes \Sc_{(3,3,1)}\QQ\dual\rightarrow \bigwedge^3 V^*\otimes \Sc_{(3,3)}\QQ\dual\rightarrow$$
$$\rightarrow V^*\otimes \Sc_{(3,2)}\QQ\dual(1)\rightarrow \bigwedge^2 V^*\otimes \Sc_{(3,1)}\QQ\dual(1)\rightarrow \bigwedge^3 V^*\otimes S^3\QQ\dual(1)\rightarrow$$
$$\rightarrow V^*\otimes S^2\QQ\dual(2)\rightarrow \bigwedge^2 V^*\otimes \QQ\dual(2)\rightarrow \bigwedge^3 V^*\otimes \OO(2)\rightarrow \OO(3)\rightarrow 0$$

If we use the relation $\Sc_{(\lambda_1+r,\lambda_2+r,\lambda_3+r)}\QQ\dual\approxeq \Sc_{(\lambda_1,\lambda_2,\lambda_3)}\QQ\dual(-r)$ (since $\Sc_{(r,r,r)}\QQ\dual\approxeq \OO(-r)$) and several duality conditions such as:
\begin{itemize}
	\item $\QQ\dual(1)\approxeq \Sc_{(1,1)}\QQ$
	\item $S^2\QQ\dual(2)\approxeq \Sc_{(2,2)}\QQ$
	\item $\Sc_{(3,1)}\QQ\dual(3)\approxeq \Sc_{(3,2)}\QQ$
	\item $S^3\QQ\dual(3)\approxeq \Sc_{(3,3)}\QQ$
\end{itemize}
we obtain the following extension:
$$0\rightarrow \OO(-3)\rightarrow V^*\otimes \QQ(-3)\rightarrow\bigwedge^2 V^*\otimes S^2\QQ(-3)\rightarrow \bigwedge^3 V^*\otimes S^3\QQ(-3)\rightarrow$$
$$\rightarrow V^*\otimes \Sc_{(3,1)}\QQ(-2)\rightarrow \bigwedge^2 V^*\otimes \Sc_{(3,1)}\QQ\dual(1)\rightarrow \bigwedge^3 V^*\otimes S^3\QQ\dual(1)\rightarrow$$
$$\rightarrow V^*\otimes S^2\QQ\dual(2)\rightarrow \bigwedge^2 V^*\otimes \QQ\dual(2)\rightarrow \bigwedge^3 V^*\otimes \OO(2)\rightarrow \OO(3)\rightarrow 0$$

This extension is not nice enough since some of the bundles that appear has intermediate cohomology. Specifically, $H^3(S^3\QQ\dual(-1))\neq 0$ and $H^6(S^3\QQ(-5))\neq 0$. 
However, one can give a splitting criterion as giving for the case of Grassmannian of lines.

Nevertheless, we propose a new extension in $H^{9}(\OO(-6))$, which seems to exist, since the Chern classes fit perfectly. Its components are now simpler vector bundles without intermediate cohomology, and it is also self-dual.
$$0\rightarrow \OO(-3)\rightarrow \bigwedge^3 V\otimes \OO(-2)\rightarrow \bigwedge^2 V\otimes \QQ(-2)\rightarrow V\otimes S^2\QQ(-2)\rightarrow$$
$$\rightarrow \bigwedge^2 V^*\otimes S^2\QQ(-1)\rightarrow \bigwedge^3 V^*\otimes \Sc_{(2,1)}\QQ(-1)\rightarrow \bigwedge^2 V\otimes S^2\QQ\dual(-1)\rightarrow$$
$$\rightarrow V^*\otimes S^2\QQ\dual\rightarrow \bigwedge^2 V^*\otimes \QQ\dual(2)\rightarrow \bigwedge^3 V^*\otimes \OO(2)\rightarrow \OO(3)\rightarrow 0$$

There exist more duality conditions, for instance:
\begin{itemize}
	\item $\Sc_{(2,1)}\QQ\dual(2)\approxeq \Sc_{(2,1)}\QQ\approxeq\Sc_{(3,2,1)}\QQ(-1)$
	\item $\Sc_{(4,2)}\QQ\dual(4)\approxeq \Sc_{(4,2)}\QQ$
	\item $\Sc_{(4,1)}\QQ\dual(4)\approxeq \Sc_{(4,3)}\QQ$
\end{itemize}
In the following proposition we generalize this type of duality conditions for a general Grassmannian. 
}
\end{example}
\begin{prop}\label{rel_dualidad}
Let $\QQ\dual$ be the universal bundle of $\G(k,n)$ with rank $k+1$. Consider $\lambda$ a Young diagram with $\lambda_1$ columns. Let us fill the rectangle $(k+1)\times \lambda_1$ and call the complement of $\lambda$ in the rectangle (up to mirror) by $\alpha$. Then:
$$\Sc_{\lambda}\QQ\dual(\lambda_1)\approxeq \Sc_{\alpha}\QQ$$
\end{prop}
\begin{proof}
By Section $1.8$ of \cite{KAP} we have the duality relation for vector spaces, hence $\Sc_{\lambda}\QQ\dual(d)\approxeq \Sc_{\alpha}\QQ$ for a twist $d$. In particular $\OO(-d)$ is a direct summand of $\Sc_{\lambda}\QQ\dual\otimes \Sc_{\alpha}\QQ\dual$. To compute $d$ we use Littlewood-Richardson's rule (see (6.7) of \cite{FH}) to this tensor product. We want to see that the only direct summand of $\Sc_{\lambda}\QQ\dual\otimes \Sc_{\alpha}\QQ\dual$ with rank $1$ is $\OO(-d)$.

First let us compute when a Schur functor $\Sc_\nu \QQ\dual$ where $\nu=(\nu_1,\ldots,\nu_r)$, with $r\leq k+1$ has rank $1$ (considering $\nu_{r+1}=\ldots=\nu_{k+1}=0$). Now let us use the formula that calculates the rank of $\Sc_\nu \QQ\dual$ (see Theorem $6.3 (1)$ of \cite{FH}):
$$rank(\Sc_\nu \QQ\dual)=\prod_{1\leq i< j\leq k+1}\frac{\nu_i-\nu_j+j-i}{j-i}$$ 
If we want this expression to be equal to $1$, each of the quotient $\frac{\nu_i-\nu_j+j-i}{j-i}$ must be $1$ since $\nu_i-\nu_j>0$ and $j-i>0$ always. This means that $\nu_i=\nu_j$ for all $1\leq i< j\leq k+1$. Therefore, the length of $\nu$ has to be $k+1$ and all the $\nu_i$ has to be equals. In that case $\Sc_\nu\QQ\dual\approxeq \OO(-\nu_1)$.

Now, by Littlewood-Richardson's rule and the previous observation, the only $\alpha-$expansion of $\lambda$ that gives a line bundle has to have length $k+1$ and all the entries equal. Consequently, the only possibility is that all these entries are equal to $\lambda_1$.
\end{proof}

\begin{example}{\rm
As an example we compute the previous relation for $\Sc_\lambda
\QQ\dual$ with $\lambda=(4,1)$ for $\G(2,5)$. We have to fill the $3\times 4$ rectangle (with yellow boxes). Finally, we mirror the complement of $\lambda$ and call the result $\alpha$, that corresponds with the partition $(4,3)$.
\begin{equation*}
	\ytableausetup
	{boxsize=1.25em}
	\ytableausetup
	{aligntableaux=top}
	\lambda= \begin{ytableau}
				{} & {} & {} & {}\\
				{} 				
				\end{ytableau}
		\quad\quad
	\begin{ytableau}
			{} & {} & {} & {}\\
			{} & *(yellow) & *(yellow) & *(yellow)\\
			*(yellow) & *(yellow) & *(yellow) & *(yellow)
			\end{ytableau}
	\quad\quad
	\alpha=\begin{ytableau}
				*(yellow) & *(yellow) & *(yellow) & *(yellow)\\
				*(yellow) & *(yellow) & *(yellow)
				\end{ytableau}
	\end{equation*}
Since the number of columns of $\lambda$ is $4$ the duality relation is:
$$\Sc_{(4,1)}\QQ\dual(4)\approxeq \Sc_{(4,3)}\QQ$$   }
\end{example}

\begin{remark}{\rm
With the same arguments used in Example \ref{mejoraFonarev}, it seems that one can extend this type of sequence for the case of $\G(k,n)$ by gluing some particular sequences that are self-dual. For $i=0,\ldots,k$, A. Fonarev give the extension of length $n-k$ from $\Sc_{(n-k,\stackrel{i+1}{\ldots},n-k)}\QQ\dual$ to $\Sc_{(n-k,\stackrel{i}{\ldots},n-k)}\QQ\dual(1)$. For a general $i$ the pieces we need seem to be the following:
\begin{equation*}
0\longrightarrow \Sc_{(n-k,\stackrel{i+1}{\ldots},n-k)}\QQ\dual \longrightarrow \bigwedge^{n-i}V\otimes \Sc_{(n-k-1,\stackrel{i+1}{\ldots},n-k-1)}\QQ\dual \longrightarrow 
\end{equation*}
\begin{equation*}
\longrightarrow \bigwedge^{n-1-i}V\otimes \Sc_{(n-k-1,\stackrel{i}{\ldots},n-k-1,n-k-2)}\QQ\dual\longrightarrow\ldots\longrightarrow \bigwedge^{k+2-i}V\otimes \Sc_{(n-k-1,\stackrel{i}{\ldots},n-k-1,1)}\QQ\dual\longrightarrow
\end{equation*}
\begin{equation*}\label{new sequence}
\longrightarrow \bigwedge^{k+1-i} V\otimes \Sc_{(n-k-1,\stackrel{i}{\ldots},n-k-1)}\QQ\dual\longrightarrow\Sc_{(n-k,\stackrel{i}{\ldots},n-k)}\QQ\dual(1)\longrightarrow 0\tag{$C_i$}
\end{equation*}
Glueing the $k+1$ pieces corresponding to $C_k, C_{k-1},\ldots,C_1,C_0$ we obtain the Serre extension we are looking for (of length $(n-k)\times (k+1)$).
More precisely, the $i-$th extension $C_i$ should be, up to a twist, the dual of the $(k+1-i)-$th extension $C_{k+1-i}$. Furthermore, $\bigwedge^i V\approxeq\bigwedge ^{n+1-i}V^*$.

Using Proposition \ref{rel_dualidad} to the vector bundles that appear in $C_i$ one can observe that:
\begin{equation*}
\Sc_{(n-k,\stackrel{i+1}{\ldots},n-k)}\QQ\dual(n-k)\approxeq \Sc_{(n-k,\stackrel{k-i}{\ldots},n-k)}\QQ
\end{equation*}  }
\end{remark}

\end{document}